\documentclass[a4paper,12pt,]{article}
\usepackage[english]{babel}
\usepackage[latin1]{inputenc}

\usepackage{amsmath,amscd}
\usepackage{amssymb}
\usepackage{amsthm}
\usepackage[arrow, matrix, curve]{xy}

\usepackage{pdfsync}

\usepackage{tikz}
\usetikzlibrary{patterns}
\usepackage{tkz-euclide}
\usetkzobj{all}

\usepackage{MnSymbol}
\usepackage{oldgerm}
\usepackage{euscript}
\usepackage{setspace} 
\usepackage[centering,bindingoffset=-1cm,headheight=15pt,includeheadfoot,textheight=22cm,textwidth=14cm]{geometry}
\usepackage[german]{babelbib}
\usepackage[flushmargin,hang]{footmisc}
\usepackage{fancyhdr}
\usepackage{paralist}

\newcommand{\GL}{\text{\textup{GL}}}
\newcommand{\Hom}{\text{\textup{Hom}}}

\newcommand{\tr}{\text{\textup{tr}}}
\newcommand{\vol}{\text{\textup{vol}}}

\newcommand{\Id}{\text{\textup{Id}}}
\newcommand{\spec}{\text{\textup{spec}}}

\newcommand{\Ad}{\operatorname{Ad}}

\makeatletter
\def\imod#1{\allowbreak\mkern10mu({\operator@font mod}\,\,#1)}
\makeatother

\theoremstyle{plain}
  \newtheorem{lemma}{Lemma}[section]
  \newtheorem{satz}[lemma]{Theorem}
  \newtheorem{propo}[lemma]{Proposition}

\theoremstyle{definition}
  \newtheorem{defn}[lemma]{\upshape Definition}

\title{A trace formula for non-unitary representations of a uniform lattice}
\author{Anton Deitmar \and Frank Monheim}
\date{}
\newcommand{\Addresses}{{
  \bigskip
  \footnotesize

  Anton Deitmar\\
  \textsc{Mathematisches Institut\\
  Auf der Morgenstelle 10\\
  72076 T\"ubingen\\
  Germany}\par\nopagebreak
   \texttt{deitmar@uni-tuebingen.de}

  \medskip

  Frank Monheim\\
  \textsc{Mathematisches Institut\\
  Auf der Morgenstelle 10\\
  72076 T\"ubingen\\
  Germany}\par\nopagebreak
  \texttt{frank.monheim@uni-tuebingen.de}

}}

\begin{document}

\maketitle
\begin{abstract}
	 In this work we shall generalize the Selberg trace formula to a non-unitary finite-dimensional complex representation $\chi: \Gamma \rightarrow \GL(V)$ of a uniform lattice $		\Gamma$ of a real Lie group $G$. 
\end{abstract}
 
\tableofcontents

\section{Introduction}
	Given a Lie group $G$, a uniform lattice $\Gamma \subset G$ and a finite dimensional unitary representation $\chi : G \rightarrow \GL(V)$, consider the Hilbert space 			$L^2(\Gamma\backslash G ,\chi)$ of measurable functions $f: G\rightarrow V$ that satisfy $f(\gamma g) = \chi(\gamma)f(g)$ for all $\gamma\in\Gamma$ and almost 			everywhere in $g\in G$, with the extra condition, that
	\begin{equation*}
  		\int_{\Gamma\backslash G} \langle f(g), f(g) \rangle\;dg < \infty. 
	\end{equation*}
	The right-regular representation on $L^2(\Gamma\backslash G,\chi)$ is easily seen to be unitary, and it is known that the right-regular representation  decomposes discretely,
	\begin{equation}\label{directdecomp}
  		R = \widehat{\bigoplus}_{\pi\in \widehat{G}}m(\pi)\pi,
	\end{equation}
	with finite multiplicities $m(\pi)\in \mathbb{N}$. This decomposition allows us to derive the Selberg trace formula: For $f\in C^\infty_c(G)$ we consider the operator $R(f)$ on $L^2(\Gamma\backslash G,\chi)$ given by
	\begin{equation*}
  		R(f)\varphi(x) = \int_G f(g)\varphi(xg)\;dx,
	\end{equation*}
	where $\varphi\in L^2(\Gamma\backslash G,\chi)$ and $dx$ is a $G$-invariant Radon measure on $\Gamma\backslash G$. The operator $R(f)$ is an integral operator, with integral 	kernel given by 
	\begin{equation*}
  		k_f(x,y) = \sum_{\gamma\in\Gamma} f(x^{-1}\gamma y)\chi(\gamma). 
	\end{equation*}
	The operator $R(f)$ is of trace class and the trace can be computed as the integral
	\begin{equation*}
  		\tr\; R(f) = \int_{\Gamma\backslash G} \sum_{\gamma\in\Gamma} f(x^{-1}\gamma x)\tr\;\chi(\gamma)\;dx.
	\end{equation*}
	After some computation and reordering this sum with respect to the conjugacy classes $[\gamma]$ of $\Gamma$ one obtains
	\begin{equation}\label{geometricside}
		\tr\; R(f) = \sum_{[\gamma]}\int \vol(\Gamma_\gamma\backslash G_\gamma)\mathcal{O}_\gamma(f) \tr\;\chi(\gamma).
	\end{equation}
	In the above expression $\Gamma_\gamma, G_\gamma$ is the centralizer of $\gamma$ in $\Gamma$ and $G$ respectively. The orbital integral $\mathcal{O}_\gamma(f)$ is 	given by
	\begin{equation*}
		\mathcal{O}_\gamma(f) = \int_{G_\gamma\backslash G}f(x^{-1}\gamma x)\;dx. 
	\end{equation*} 
	On the other hand we can compute the trace of $R(f)$ using the decomposition in (\ref{directdecomp}) and obtain
	\begin{equation}\label{spectralside}
		\tr\;R(f) = \sum_{\pi\in\widehat{G}}m(\pi)\tr\;\pi(f). 
	\end{equation}
	Thus, by equating (\ref{geometricside}) and (\ref{spectralside}) one obtains the Selberg trace formula
	\begin{equation*}
		 \sum_{\pi\in\widehat{G}}m(\pi)\tr\;\pi(f) = \sum_{[\gamma]}\int \vol(\Gamma_\gamma\backslash G_\gamma)\mathcal{O}_\gamma(f) \tr\;\chi(\gamma),
	\end{equation*}
	which connects spectral data (left-hand side) to geometric data (right-hand side) of $\Gamma\backslash G$. \\

	In contrast to the above explanations, we will in this work allow $\chi : G\rightarrow \GL(V)$ to be an arbitrary complex and finite dimensional representation, not necessarily 	unitary. If we follow the approach, which we have sketched for a unitary representation $\chi$, the first obstacle appearing, is the question how to define $L^2(\Gamma			\backslash G,\chi)$. A satisfying answer will be given in Definition \ref{hilbertspace}. But nevertheless, we will lose the unitarity of the right-regular representation $R$ and 		consequently, a decomposition (\ref{directdecomp}) as in the unitary case will in general not hold. 
	But under the conditions, that will be formulated in the subsequent section, it is shown in Proposition \ref{l2filtration}  that the right-regular representation on $L^2(\Gamma		\backslash G,\chi)$ admits an increasing and exhausting filtration
	\begin{equation}\label{fil}
    		0 = V_0 \subset V_1\subset \dots \subset \bigcup_{i=0}^\infty V_i = L^2(\Gamma\backslash G,\chi),
  	\end{equation}
	such that the natural representation on $V_i/V_{i-1}$ is admissible and irreducible. 
	
	The above filtration gives the means to derive a trace formula in Theorem 			
	\ref{traceformula}. The computation of the geometric side is completely analogous to the unitary case, whereas the spectral side can be computed using the filtration (\ref{fil}). 	We use the fact, that the trace of an operator on $L^2(\Gamma\backslash G,\chi)$, which is compatible with the filtration, can be computed by summing the individual traces of 	the induced operators on the graded parts $V_i/V_{i-1}$.

\section{Preliminaries} 

In this work, let $G$ be a connected, semisimple Lie-group with finite center and $K$ a maximal compact subgroup of $G$. Let $\Gamma\subset G$ be a torsion free, uniform lattice. Fix a Haar measure $dg$, which is both a left- and right Haar measure, due to the existence of the uniform lattice $\Gamma$.  We let $X := \Gamma\backslash G$ be the compact quotient space with $G$-invariant measure $dx$, such that for $f\in C_c(G)$ the integral formula
\begin{equation*}
  \int_G f(g)\;dg = \sum_{\gamma\in\Gamma} \int_X f(\gamma x)\;dx
\end{equation*}
holds.

\begin{defn}
  \begin{itemize}
    	\item For a complex vector space $V$ we let $\GL(V)$ be the group of all automorphisms of $V$. A \emph{representation} $(\pi,V)$ of $G$ is a group homomorphism $\pi : G  
    	\rightarrow \GL(V)$.

	\item	For a complex Hilbert space we let $\GL(H)$ be the group of all bijective and bounded endomorphisms of $H$. A \emph{(continuous) representation} $(\pi,H)$ of $G$ is 	a group homomorphism $\pi: G\rightarrow \GL(H)$, such that the map

	\begin{equation*}
  		\begin{split}
    		G\times H &\rightarrow H,\\
    		(g,v)&\mapsto \pi(g)v,
 		 \end{split}
	\end{equation*}

	is continuous. A continuous representation $\pi$ is said to be \emph{admissible} if $\pi$ restricted to $K$ is unitary and each $\tau\in \widehat{K}$ occurs with only finite 		multiplicity. When the underlying representation space is a Hilbert space we will always mean a continuous representation without mentioning it anymore. 

	\item A \emph{(Lie algebra) representation} $(\pi,V)$ of the Lie algebra $\mathfrak{g}$ of $G$ is a complex vector space $V$ together with a Lie algebra homomorphism

	\begin{equation*}
  	\begin{split}
    		\mathfrak{g} & \rightarrow \mathfrak{gl}(V),\\
    		X &\mapsto \pi(X).
 	 \end{split}
	\end{equation*}

	\item A $(\mathfrak{g},K)$-module is a vector space $V$ which is both a Lie algebra representation of $
	\mathfrak{g}$ and a group representation of $K$ such that the representations are compatible in the following way:
	\begin{enumerate}
   		\item for any $v\in V$, $k\in K$ and $X\in\mathfrak{g}$
     		\begin{equation*}
       		k\cdot (X\cdot v) = (\Ad(k)X)\cdot(k\cdot v),
     		\end{equation*}
   		\item for any $v\in V$ and $Y\in\mathfrak{k}$
     		\begin{equation*}
       			\left.\left(\frac{d}{dt}\exp(tY)\cdot v\right)\right|_{t=0} = Y\cdot v.
     		\end{equation*}
	\end{enumerate}
	The third condition ensures $K$-finiteness:
	\begin{enumerate}
  		\item[3.] for any $v\in V$ the set $Kv$ spans a finite-dimensional subspace of $V$. 
	\end{enumerate}
	Recall that the $K$-finite vectors of an admissible representation $(\pi,V)$ give rise to a $(\mathfrak{g},K)$-module. Two admissible representations $\pi$ and $\eta$ are 		called \emph{equivalent}, if the associated $(\mathfrak{g},K)$-modules are isomorphic. 

	\item The \emph{unitary dual} $\widehat{G}$ of the group $G$ is the set of all irreducible unitary representations modulo unitary equivalence. 

	\item The \emph{admissible dual} $\widehat{G}_{\text{adm}}$ of $G$ is the set of all irreducible admissible representations module admissible equivalence. 
 	\end{itemize}
\end{defn}

In the following $\chi$ will always be a finite dimensional complex representation of $\Gamma$, not necessarily unitary, with representation space $V=V_\chi$.  
Let $E=E_\chi$ be the associated vector bundle over $\Gamma\backslash G$. More precisely, we consider the bundle $\Gamma\backslash (G\times V)$, where $\Gamma$ acts on $G\times V$ via 	
	$$\gamma\cdot(g,v)=(\gamma g, \chi(\gamma)v).$$
The image of $(g,v)$ under the canonical projection $ G\times V\rightarrow E$ will be written as $\Gamma(g,v)$. Note, that $\Gamma(\gamma g,v) = \Gamma(g,\chi(\gamma^{-1})v)$. Furthermore the group $G$ acts on $E$ via 
\begin{equation*}
  g\cdot \Gamma(h,v) = \Gamma(hg^{-1},v). 
\end{equation*}
There is a canonical identification of the smooth sections $\Gamma^\infty(X,E)$ of the bundle $E$ with the set of functions
\begin{equation*}
  C^\infty(G,V)^\Gamma = \{f\in C^\infty(G,V): f(\gamma g) = \chi(\gamma)f(g) \text{ for all } \gamma\in\Gamma,g\in G\},
\end{equation*} 
and we will freely switch between these two interpretations. 

\begin{defn}\label{hilbertspace}
We choose any smooth hermitian metric $\langle\cdot,\cdot\rangle$ on $E$. If we choose a Haar-measure $dk$ on $K$ we can form the integral
\begin{equation*}
  \int_K\left\langle \Gamma(gk,v),\Gamma(gk,w)\right\rangle_{\Gamma gk}\;dk.
\end{equation*}
 This gives again a smooth hermitian fibre metric on $E$, which is $K$-equivariant and by replacing $\left\langle\cdot,\cdot\right\rangle$ with this one we can assume, that 
\begin{equation*}
  \left\langle \Gamma(gk,v),\Gamma(gk,w)\right\rangle_{\Gamma gk} = \left\langle \Gamma(g,v),\Gamma(g,w)\right\rangle_{\Gamma g}
\end{equation*}
for arbitrary $g\in G,k\in K$ and $v,w\in V$. 

Together with this smooth $K$-equivariant metric we obtain a pre-Hilbert space structure on the set of smooth sections $\Gamma^\infty(X,E)$  via 
\begin{equation*}
  (f,g) := \int_{\Gamma\backslash G} \left\langle f(x),g(x) \right\rangle_x\;dx. 
\end{equation*}
We complete $\Gamma^\infty(X,E)$ with respect to the induced norm and  obtain the Hilbert space of square integrable sections $L^2(X,E)$. The definition of $L^2(X,E)$ is independent of the chosen smooth fibre metric on $E$, since by compactness of $\Gamma\backslash G$ and the finite-dimensionality of the fibres, another smooth metric induces an equivalent norm on $\Gamma^\infty(X,E)$. 
\end{defn}

\begin{defn}
On $\Gamma^\infty(X,E)$ we define the \emph{right regular representation} $R$ of $G$ as
\begin{equation*}
  R(g)f(x) := g\cdot f(xg),
\end{equation*}
where $f\in \Gamma^\infty(X,E), x\in X$ and $g\in G$. \\

If we use the identification of $\Gamma^\infty(X,E)$ with $C^\infty(G,V_\chi)^\Gamma$ the right regular representation for elements $f\in C^\infty(G,V_\chi)^\Gamma$ is just given by
\begin{equation*}
  R(g)f(x) = f(xg).
\end{equation*}
\end{defn}

In the following paragraph we will show, that the right regular representation $R$ is continuous. We start out with the following proposition.

\begin{propo}\label{continuousprop}
  There exists a continuous function $\psi$ on $G$, such that
  \begin{equation*}
    \left\langle \Gamma(gh,v), \Gamma(gh,v)\right\rangle_{\Gamma gh} \leq \psi(h)\left\langle \Gamma(g,v),\Gamma(g,v)\right\rangle_{\Gamma g}, 
  \end{equation*}
  for arbitrary $g,h\in G$ and $v\in V$. 
  In particular, there exists for each compact set $C\subset G$ a constant $M$ depending only on $C$, such that $\left\langle\Gamma(gh,v),\Gamma(gh,v) \right\rangle_{\Gamma gh} \leq M\left\langle \Gamma(g,v),\Gamma(g,v)\right\rangle_{\Gamma g}$, for all $h\in C$ and arbitrary $g\in G,v\in V$.   
\end{propo}

\begin{proof}
  For every $h\in G$, there exists a contionuous section $A_h \in \Gamma(X,\Hom(E,E))$, such that for $v\in V$ we have
  \begin{equation*}
    \left\langle\Gamma(gh,v),\Gamma(gh,v)\right\rangle_{\Gamma gh} = \left\langle\Gamma(g,A_h(\Gamma g)v),\Gamma(g,A_h(\Gamma g)v)\right\rangle_{\Gamma g}.
  \end{equation*}
  As the metric is smooth, the dependence of $A_h$ on $h$ is smooth, in particular continuous.
  We let 
  \begin{equation*}
    \|A_h(\Gamma g)\|_{\Gamma g}^2 = \sup_{v\neq 0}\frac{\left\langle\Gamma(g,A_h(\Gamma g)v),\Gamma(g,A_h(\Gamma g)v)\right\rangle_{\Gamma g}}{\left\langle\Gamma(g,v),\Gamma(g,v)\right\rangle_{\Gamma g}}.
   \end{equation*}
    Then we define $\psi(h) = \max_{\Gamma g \in \Gamma\backslash G}\|A_h(\Gamma g)\|^2_{\Gamma g}$ which satisfies the conditions of the proposition. 
\end{proof}

We choose once and for all a representative $(\tau,V_\tau)$ for each class in $\widehat{K}$, the unitary dual of $K$. For a unitary representation $(\pi, V)$ of $K$ we let $V(\tau)$ be the $\tau$-isotype. Recall the following theorem:

\begin{satz}\emph{\cite[Theorem 7.3.2.]{prohadeitmar}}
  For $(\pi,V)$ a unitary representation of $K$ the representation space is a direct Hilbert space sum of all the $K$-isotypes:
  \begin{equation*}
    V = \widehat{\bigoplus}_{\tau\in\widehat{K}} V(\tau). 
  \end{equation*}
\end{satz}

\begin{propo}
  The right regular representation $R$ on $\Gamma^\infty(X,E)$ is continuous with respect to the $L^2$-topology. In particular it extends to a continuous representation of $L^2(X,E)$. The restriction of $R$ to the maximal compact subgroup $K$ is unitary and hence we get a $K$-isotypical decomposition
  \begin{equation*}
    L^2(X,E) = \widehat{\bigoplus}_{\tau\in\widehat{K}}L^2(X,E)(\tau).
  \end{equation*}
  \end{propo}

\begin{proof}
  Let $C\subset G$ be a compact subset and $h\in C$. Let $M$ be as in Proposition \ref{continuousprop}. Then for $f\in \Gamma^\infty(X,E)$ we estimate
  \begin{align*}
    (R(h)f,R(h)f) & = \int_X \left\langle h\cdot f(xh), h\cdot f(xh)\right\rangle_x \;dx\\
                           & \leq \int_X M\left\langle f(x),f(x) \right\rangle_{x}\;dx\\
                           &= M(f,f),
  \end{align*}
  independent of $h\in C$. Hence the operator norm of $R(h)$ is uniformally bounded on each compact subset $C\subset G$. Since for fixed $f\in \Gamma^\infty(X,E)$ the map 
  \begin{align*}
    G & \rightarrow L^2(X,E)\\
    g & \mapsto R(g)f
  \end{align*} 
  is continous, the continuity of the right-regular representation $R$ follows. 
  The unitarity as a representation of $K$ is clear from the $K$-equivariance of the fibre metric: If we let $k\in K$ then we get
  \begin{align*}
  (R(k)f,R(k)f) & = \int_X \left\langle k\cdot f(xk), k\cdot f(xk)\right\rangle_x \;dx\\
                           & = \int_X \left\langle f(xk),f(xk)\right\rangle_{xk}\;dx\\
                           & = \int_X \left\langle f(x),f(x)\right\rangle_{x}\;dx\\
                           & = (f,f).
  \end{align*}
\end{proof}

\section{The spectral decomposition of the Casimir operator}
\begin{defn} 
  We let $\mathfrak{g}, \mathfrak{g}_\mathbb{C}, U(\mathfrak{g}_\mathbb{C}), Z(\mathfrak{g}_\mathbb{C})$ be the Lie-algebra of $G$, its complexification, the universal envelopping algebra, as well as its center, respectively. We endow $G/K$ with the $G$-invariant metric induced by the Killing form $\langle\cdot,\cdot\rangle$. The Killing form is non-degenerate, and hence it gives an identification of $\mathfrak{g}$ and its dual space $\mathfrak{g}^*$. If $X_1,\dots,X_n$ is a basis of $\mathfrak{g}$, then we define the \emph{Casimir element} as
  \begin{equation*}
    \Omega = Y_1X_1 +\dots Y_nX_n \in U(\mathfrak{g}_\mathbb{C}),
  \end{equation*}
  where $Y_1, \dots, Y_n$ is a dual basis of $X_1,\dots, X_n$ with respect to the Killing form. The definition of $\Omega$ is independent of the chosen orthonormal basis and $\Omega\in Z(\mathfrak{g}_\mathbb{C})$ \cite[Proposition 8.6.]{knapp}.
  
   If $\mathfrak{g}=\mathfrak{k}\oplus\mathfrak{p}$ is the Cartan decomposition of $\mathfrak{g}$ and $\theta$ is the Cartan involution on $\mathfrak{g}$ the map
  \begin{equation*}
    (X,Y) \mapsto - \langle X,\theta(Y)\rangle,
  \end{equation*}
  is a positive definite bilinear form.
   If $X_1,\dots, X_l$ is an orthonormal basis of $\mathfrak{k}$ and $Y_1,\dots,Y_k$ an orthonormal basis of $\mathfrak{p}$ with respect to the above bilinear form, we find
  \begin{align*}
    \Omega &=  -X^2_1 -\dots  -X^2_l + Y_1^2 + \dots + Y_k^2\\
    & = \Omega_K + Y_1^2 + \dots + Y_k^2,\\
  \end{align*}
  where $\Omega_K$ is the Casimir element of $U(\mathfrak{k}_\mathbb{C})$. 
\end{defn}

We will need the following results, to deduce a nice spectral decomposition of the Casimir operator on $L^2(X,E)$. 

\begin{satz}\label{shubin}\emph{\cite[Theorem 8.4.]{shubin}}
Let $M$ be a closed manifold and $D$  an elliptic differential operator on a metric bundle $\mathcal{E}$. If the resolvent set $\rho(D)\neq\emptyset$ is not empty, the spectrum $\sigma(D)$ is discrete and for each $\lambda\in\sigma(D)$ there exists a decomposition $L^2(M,\mathcal{E})=E_\lambda\oplus E'_\lambda$, such that
\begin{enumerate}
  \item $E_\lambda\subset \Gamma^\infty(M,\mathcal{E}), \dim E_\lambda<\infty$, and $E_\lambda$ is invariant under $D$ and there exists some $n>0$, such that $(D-\lambda)^nE_\lambda = 0$,
  \item $E'_\lambda$ is a closed subspace of $L^2(M,\mathcal{E})$ invariant under $D$. If we denote by $D_\lambda$ the restriction of $D$ to $ E'_\lambda$, then $\lambda\not\in\sigma(A_\lambda)$.  
\end{enumerate}
\end{satz}

\begin{satz}\label{shubin2}\emph{\cite[Theorem 8.4.]{shubin}}
  Let $M$ be a closed manifold and $D$  an elliptic differential operator on a metric bundle $\mathcal{E}$. For an interval $I$ we define the cone
  \begin{equation*}
    \Lambda_I = \{re^{i\theta} : 0 \leq r < \infty, \theta\in I\}.  
  \end{equation*}
  For $\varepsilon > 0$ there exists an $R>0$ such that the spectrum $\sigma(D)$ is contained in the set $B_R(0)\cup \Lambda_{[-\varepsilon,\varepsilon]}$. 
\end{satz}

\begin{defn}
  Let $\mathcal{E}$ be a vector bundle over a Riemannian manifold $M$. A second order differential operator $D$ is a \emph{Laplace type operator} if for the principal symbol 
  \begin{equation*}
    \sigma_2(D)(x,\xi)=\|\xi\|^2,
  \end{equation*}
  for arbitrary $x\in M$ and $\xi \in T^*M$. In particular each Laplace type operator is elliptic. 
\end{defn}

The reason why we introduce the notion of Laplace type operators, is that a reasonable spectral theory can be developed for those. 

\begin{propo}\label{casimir}
  The Casimir element $\Omega$ induces on each $K$-isotype $L^2(X,E)(\tau)$ a Laplace type operator, having discrete spectrum. We will denote the induced operator by $\Omega_\tau$. Let
  \begin{equation*}
    V_{\tau,\lambda} := \{f \in L^2(X,E)(\tau) : (\Omega_\tau - \lambda)^n f =0 \text{ for some } n\in\mathbb{N}\},
  \end{equation*}
  the generalized eigenspace belonging to $\lambda\in\spec(\Omega_\tau)$. Then $V_{\tau,\lambda}\subset \Gamma^\infty(X,E)(\tau)$, $\dim V_{\tau,\lambda} < \infty$ and $V_{\tau,\infty}$ is stable under $K$ as well as $Z(\mathfrak{g}_\mathbb{C})$. 
\end{propo}
\begin{proof}
  We let $E_{\chi,\tau}$ be the vector bundle $G \times_{\Gamma\times K} V_\chi\otimes V_\tau$ with base space $\Gamma\backslash G/K$ where $\Gamma\times K$ acts on $G\times V_\chi\otimes V_\tau$ via
  \begin{equation*}
    (\gamma,k)\cdot(g,v\times w) = (\gamma g k^{-1}, \chi(\gamma)v\times \tau(k)w).
  \end{equation*}
  Let $X\in\mathfrak{g}$ and $f\in C^\infty(G,V_\chi\times V_\tau)$. $X$ induces a differential operator via 
  \begin{equation*}
    Xf(g) = \left.\frac{d}{dt}f(g\exp(tX))\right|_{t=0},
  \end{equation*}
  and this map from the Lie algebra to the algebra of differential operators extends to the universal envelopping algebra $U(\mathfrak{g}_\mathbb{C})$. From the definition it is easily seen, that the induced differential operator is left invariant, meaning that
  \begin{equation*}
    X(L_h(f))(g) = (Xf)(hg),
  \end{equation*}
  where $L_h$ is the left translation by the element $h$. On the other hand, we find for $X\in \mathfrak{g}$ and the right translation $R_h$ and any element $h\in G$, that
  \begin{equation*}
\begin{split}
  X(R_h(f))(g) &= \left.\frac{d}{dt}f(g\exp(tX)h)\right|_{t=0}\\
                               &= \left.\frac{d}{dt}f(gh\exp(\text{Ad}(h)tX))\right|_{t=0}\\
                               &= (\text{Ad}(h)X)(f)(gh).\\
\end{split}
\end{equation*}
Thus we get $X(R_h(f))(g)=(\text{Ad}(h)X)(f)(gh)$ for all $X\in U(\mathfrak{g}_\mathbb{C})$. Since $\Omega\in Z(\mathfrak{g}_\mathbb{C})$ this yields
\begin{equation*}
  \Omega(R_h(f))(g) = (\Omega f)(gh),
\end{equation*}
  for arbitrary $g,h\in G$ and consequently we have for $f\in C^\infty(G, V_\chi\otimes V_\tau)^{\Gamma\times K}$ that
    \begin{equation*}
  (\Omega f)(\gamma xk^{-1}) = \chi(\gamma)\otimes\tau(k)\Omega f(x)
  \end{equation*}
  and thus $C^\infty(G, V_\chi\otimes V_\tau)^{\Gamma\times K}$ is stable under $\Omega$. 
   
   Furthermore
   \begin{equation*}
     \Gamma^\infty(X,E)(\tau) \cong V_\tau\otimes \Hom_K(V_\tau,\Gamma^\infty(X,E)),
   \end{equation*}
  and
  \begin{align*}
    \Hom_K(V_\tau,\Gamma^\infty(X,E)) &\cong (\Gamma^\infty(X,E)\otimes V_\tau)^K\\
    &\cong (C^\infty(G)\otimes V_\chi\otimes V_\tau)^{\Gamma\times K}\\
    &\cong \Gamma^\infty (\Gamma\backslash G/K,E_{\chi,\tau}).
  \end{align*}
  Hence $\Id\otimes\Omega$ induces an operator $\Omega_\tau$ on $L^2(E)(\tau)$. \\
  
  Next we will show, that it is a Laplace type operator. 
  
  Consider the Cartan decomposition $\mathfrak{g} = \mathfrak{k} \oplus \mathfrak{p}$ and basises $X_1,\dots X_l$ and $Y_1,\dots Y_k$ as before, such that
  \begin{equation*}
    \Omega = -X^2_1 -\dots  -X^2_l + Y_1^2 + \dots + Y_k^2 = \Omega_K + Y_1^2 + \dots + Y_k^2. 
  \end{equation*}
  For $\Gamma gK \in \Gamma\backslash G/K$ there exists a neighbourhood $U$, such that we can choose the map
  \begin{equation*}
    g\exp(y_1Y_1+\dots+ y_kY_n) \mapsto (y_1,\dots,y_k)
  \end{equation*}
  as a local coordinate map on $U$. Inside $U$ and with these coordinates we find
  \begin{equation*}
    (\Omega f)(g) = \sum\frac{\partial^2}{\partial y_i^2}f(g)+\Id\otimes\sum(d\tau)(\Omega_K) f(g).
  \end{equation*}
  But since the representation $V_\tau$ is irreducible  and $\Omega_K\in Z(\mathfrak{k}_\mathbb{C})$ the operator $(d\tau)(\Omega_K)$ acts as a scalar, according to the Lemma of Schur. Consequently $\Omega$ induces a second order differential operator with principal symbol $(\Omega_\tau)(x,\xi) = \|\xi\|^2$.  
  
  From Theorem \ref{shubin2} and \ref{shubin} it is now clear, that the spectrum is discrete. By Theorem \ref{shubin} it follows also, that the generalized eigenspace $V_{\tau,\lambda}$ is finite dimensional and $V_{\tau,\lambda} \subset\Gamma^\infty(X,E)(\tau)$.

 $V_{\tau,\lambda}$ is stable under $K$, since $K\cdot L^2(X,E)(\tau) \subset L^2(X,E)(\tau)$ and $\text{Ad}(k)\Omega = \Omega$. Since $Z(\mathfrak{g}_\mathbb{C})\cdot \Gamma^\infty(X,E)(\tau) \subset \Gamma^\infty(X,E)(\tau)$ it is also clear, that $Z(\mathfrak{g}_\mathbb{C})\cdot V_{\tau,\lambda}\subset V_{\tau,\lambda}$. 
\end{proof}

\begin{propo}\label{unprooven}
  The space $L^2(X,E)(\tau)$ is the closure of the algebraic direct sum of all generalized eigenspaces:
  \begin{equation*}
    L^2(X,E)(\tau) = \overline{\bigoplus_{\lambda\in\sigma(\Omega_\tau)}V_{\tau,\lambda}}.
  \end{equation*}
\end{propo}

To prepare the proof of  Proposition \ref{unprooven} we need:

\begin{defn}
  Let $H$ be a Hilbert space and $G$ a linear operator with non-empty resolvent set $\rho(G) \neq \emptyset$. An operator $B$ is said to be compact relative to $G$ if $D(G)\subset D(B)$ and the operator $BR_\lambda(G)$ is compact, where $R_\lambda(G)=(G-\lambda)^{-1}$ is the resolvent of $G$. 
\end{defn}

\begin{satz}\label{relkomp}\emph{\cite[Theorem 4.3.]{markus}}
 Let $H$ be a Hilbert space and $G$ a self-adjoint operator. The resolvent $R_\lambda(G)$ is assumed to be a Schatten class operator and $B$ an operator relatively compact to $G$. Then the operator $C = G + B$  has a compact resolvent and $H$ is the closure of the generalized eigenspaces of $C$:
 \begin{equation*}
  H = \overline{\bigoplus_{\lambda\in\sigma(C)}V_\lambda}
 \end{equation*}
\end{satz}

\begin{proof}[Proof of Proposition \ref{unprooven}]
According to Proposition \ref{casimir} we have $\sigma(\Omega_\tau)(x,\xi)=\|\xi\|^2$. Hence we get $\Omega_\tau = \Delta + B$, where $\Delta$ is the Bochner-Laplace operator, which is self-adjoint, and $B$ is a first order differential operator. The resolvent $R_\lambda(\Delta)$ is of order $-2$, hence compact and a Schatten class operator. Similarly $BR_\lambda(\Delta)$ is of order $-1$ and also compact. The statement now follows by applying Theorem \ref{relkomp}. 
\end{proof}

\begin{defn}
  Let $V_{\text{fin}} \subset \Gamma^\infty(X,E)$ be the set of all smooth sections which are $K$- as well as $Z(\mathfrak{g}_\mathbb{C})$-finite. 
\end{defn}
\begin{propo}
  $V_{\text{fin}}$ is the algebraic direct sum of all generalized eigenspaces for the operators $\Omega_\tau$:
  \begin{equation}\label{Vfindirect}
    V_{\text{fin}} = \bigoplus_{\tau\in\widehat{K}}\bigoplus_{\lambda\in\sigma(\Omega_\tau)}V_{\tau,\lambda}. 
  \end{equation}
  In particular, $V_{\text{fin}}$ is dense in $L^2(X,E)$ and consequently, because $V_{\text{fin}}\subset\Gamma^\infty(X,E)$, it is dense in $\Gamma^\infty(X,E)$. 
\end{propo}

\begin{proof}
  If $f$ is an element in the above direct sum, it is clear that $f\in V_{\text{fin}}$, since each generalized eigenspace $V_{\tau,\lambda}$ is finite dimensional and $K$- and $Z(\mathfrak{g}_\mathbb{C})$-invariant.
  
  If we now take $f\in V_{\text{fin}}$ we obtain 
  \begin{equation*}
    f\in \bigoplus_{\tau\in\widehat{K}} L^2(X,E)(\tau)
  \end{equation*}
  because of the $K$-finiteness of $f$. Hence we can assume, that $f\in L^2(X,E)(\tau)$ for some $\tau\in\widehat{K}$. Let $W\subset L^2(X,E)(\tau)$ be the finite-dimensional $Z(\mathfrak{g}_\mathbb{C})$- and $K$-invariant vectorspace, generated by $f$. $W$ is stable under $\Omega_\tau$, since this operator is induced by $\Omega\in Z(\mathfrak{g}_\mathbb{C})$. Consider the operator $\Omega_\tau|_W$. By the theorem about the Jordan normal form we have a direct sum decomposition
  \begin{equation*}
    W = \bigoplus_{\lambda\in\sigma(\Omega_\tau)} V_{\tau,\lambda}\cap W.
  \end{equation*}
  This proves, that $f$ lies in the above direct sum. 
\end{proof}

We will now cite two propositions, which we will need to infer a filtration of the $(\mathfrak{g},K)$-module $V_{\text{fin}}$. 

\begin{propo}\label{wallachprop}\emph{\cite[Corollary 3.4.7.]{wallach}}
  Let $V$ a finitely generated $(\mathfrak{g},K)$-module, such that $\dim Z(\mathfrak{g_\mathbb{C}})v<\infty$ for all $v\in V$. Then $V$ is admissible. 
\end{propo}

\begin{propo}\label{knappprop}\emph{\cite[Corollary 10.42.]{knapp}}
  Each Harish-Chandra-module $V$ (in other words: a finitely generated, admissible $(\mathfrak{g}, K)$-module) has a finite composition series
  \begin{equation*}
   V = W_k\supset W_{k-1}\supset \dots \supset W_0= 0
  \end{equation*}
 with irreducible quotients $W_j/W_{j-1}$. The multiplicities of the irreducible subquotients are independent of the chosen composition series.  
\end{propo}

\begin{propo}\label{filpropo}
  The exists a seperated, exhaustive and increasing filtration $\text{Fil}_iV$, where $i$ ranges over all nonnegative integers, of $V_{\text{fin}}$ as a $(\mathfrak{g},K)$-module, such that each quotient $\text{Fil}_iV/\text{Fil}_{i-1}V$ is admissible and irreducible. The graduated module
  \begin{equation*}
    \text{Gr}\;V_{\text{fin}} = \bigoplus_{i=0}^\infty \text{Fil}_{i+1}V/\text{Fil}_iV 
  \end{equation*}
  is independent of the chosen filtration. 
  \end{propo}
  
  \begin{proof}
    We choose a generalized eigenspace $V_{\tau,\lambda}$ from the direct sum in (\ref{Vfindirect}). Since it is $K$- and $Z(\mathfrak{g}_\mathbb{C})$-stable, we find $U(\mathfrak{g}_\mathbb{C})V_{\tau,\lambda} \subset V_{\text{fin}}$. According to Proposition \ref{wallachprop} the $(\mathfrak{g},K)$-module $U(\mathfrak{g}_\mathbb{C})V_{\tau,\lambda}$ is admissible and according to Proposition \ref{knappprop} there exists a finite composition series
     \begin{equation*}
     U(\mathfrak{g}_\mathbb{C})V_{\tau,\lambda} = \text{Fil}_kV\supset \text{Fil}_{k-1}V\supset\dots\supset \text{Fil}_0V = 0
   \end{equation*}
   such that $\bigoplus_{i=1}^k \text{Fil}_iV/\text{Fil}_{i-1}V$ is independent of the chosen composition series. Now proceed in the same manner with $V_{\text{fin}}/U(\mathfrak{g}_\mathbb{C})V_{\tau,\lambda}$ to obtain the filtration. 
  \end{proof}
  
\begin{propo}\label{realanalytic}
  Each element $f\in V_{\tau,\lambda}$ is real-analytic. 
\end{propo}

\begin{proof}\label{realanalytic}
  If $f\in V_{\tau,\lambda}$ there exists some $N\in\mathbb{N}$, such that $(\Omega_\tau-\lambda)^Nf$ = 0. Hence $f$ is annihilated by the elliptic differential operator $(\Omega_\tau - \lambda)^N$, whence it is real-analytic. 
\end{proof}

The filtration of $V_{\text{fin}}$ furnishes a filtration of the right regular representation, as we will show now.

\begin{propo}\label{l2filtration}
  The representation $(R,L^2(X,E))$ admits a seperated, exhaustive and increasing filtration of subrepresentations
  \begin{equation*}
    0 = V_0 \subset V_1\subset \dots \subset \bigcup_{i=0}^\infty V_i = L^2(X,E),
  \end{equation*}
   induced by the filtration of the $(\mathfrak{g},K)$-module $V_{\text{fin}}$ from Proposition \ref{filpropo}. 
\end{propo}

\begin{proof}
  We let $V_i = \overline{\text{Fil}_iV}$ be the closure of $\text{Fil}_iV$ in $L^2(X,E)$. It is easy to see, that $V_i$ is stable under $G$. To show this, it is enough to prove $G\cdot\text{Fil}_iV\subset V_i$, since the representation is continuous. Let $h\in V_i^\perp$. For $f\in \text{Fil}_iV \subset V_{\text{fin}}$  the function
  \begin{equation*}
    g \mapsto (R(g)h,f)
  \end{equation*}
is real-analytic according to Proposition \ref{realanalytic}. For $X\in\mathfrak{g}$ on a sufficient small neighbourhood of 0, we have a Taylor expansion:
\begin{equation*}
  \begin{split}
    \left\langle R(\exp X)h,f \right\rangle &= \left.\sum_{n=0}^\infty\frac{1}{n!}X^n\left\langle R(g)h,f \right\rangle\right|_{g=1}\\
    &= \sum_{n=0}^\infty\frac{1}{n!}\left\langle X^nh,f \right\rangle.
  \end{split}
\end{equation*} 
Since $h\in \text{Fil}_iV$, hence also $X^nh\in\text{Fil}_iV$ we get $ \left\langle R(\exp X)h,f \right\rangle=0$. But then $\left\langle R(g)h,f \right\rangle=0$ in a neighbourhood of 1, and because of analycity, on the whole of $G$. Since $h\in \text{Fil}_iV$ and $f\in V_i^\perp$ were arbitrary we get $G\cdot\text{Fil}_iV\subset V_i$.
\end{proof}

For each $i\in\mathbb{N}$ we obtain the quotient representation on $V_i/V_{i-1}$, which we will denote by $R_i$. 

\begin{propo}
  For the $K$-finite vectors of $R_i$ we obtain
   \begin{equation*}
    \left(V_i/V_{i-1}\right)_{R_i,K} \cong \text{Fil}_iV/\text{Fil}_{i-1}V.
  \end{equation*}
\end{propo}

\begin{proof}
  This is is clear, because $R_i$ is admissible, and so $(V_i/V_{i-1})(\tau)$ is finite dimensional, but 
  \begin{equation*}
    (\text{Fil}_iV/\text{Fil}_{i-1}V)(\tau) \subset (V_i/V_{i-1})(\tau)
    \end{equation*}
    is dense, and thus the both must be equal. 
\end{proof}

Together with the following theorem, it follows that the representations $R_i$ are irreducible:

\begin{satz}\label{adm}\emph{\cite[Theorem 3.4.12.]{wallach}}
  Let $(\pi,H)$ be an admissible Hilbert-space representation of G. Then $(\pi,H)$ is irreducible, iff the associated $(\mathfrak{g},K)$-module $H_{\pi,K}$ is irreducible. 
\end{satz}

\section{The trace formula}

\begin{defn}
  Let $(H,\left\langle\cdot , \cdot\right\rangle)$ be a Hilbert-space. A compact operator $T$ is of \emph{trace class}, if $\sum_is_i(T)<\infty$, where $s_i(T)$ denote the singular values of the operator $T$. Let $(e_i)_{i\in I}$ be an orthonormal basis of $H$. The \emph{trace} of a trace class operator $T$ is defined as 
  \begin{equation*}
    \operatorname{tr}(T) = \sum_{i\in I}\langle Te_i, e_i\rangle.
  \end{equation*}
 One can show that this sum converges absolutely and is independent of the chosen orthonormal basis $(e_i)_{i\in I}$ (\cite[Theorem 5.3.5]{prohadeitmar}). \\
 
 Now let $(\pi,H)$ be an admissible representation of $G$. The representation is said to be of \emph{trace class} if for each $f\in C^\infty_c(G)$ the operator
 \begin{equation*}
   \pi(f) = \int_G f(g)\pi(g)\;dg
 \end{equation*}
 is of trace class. 
\end{defn}

\begin{satz}\label{traceformula}
  Let $f\in C_c^\infty(G)$. The operator $R(f)$ on $L^2(X,E)$ is an integral operator with integral kernel
  \begin{equation*}
    k_f(x,y)= \sum_{\gamma\in\Gamma}f(x^{-1}\gamma y)\chi(\gamma). 
  \end{equation*}
  Thus $R(f)$ is of trace class and the following trace formula holds
    \begin{equation*}
     \sum_{[\gamma]}\vol(\Gamma_\gamma\backslash G_\gamma)\mathcal{O}_\gamma(f)\tr\;\chi(\gamma)
 = \sum_{\pi\in\widehat{G}_{\text{adm}}} N_{\Gamma,\chi}(\pi)\tr\;\pi(f),
  \end{equation*}
  where on the left hand side we sum over all conjugacy classes $[\gamma]$ of $\Gamma$ and $\mathcal{O}_\gamma(f)$ is the orbital integral
  \begin{equation*}
    \mathcal{O}_\gamma(f)=\int_{G_\gamma\backslash G} f(x^{-1}\gamma x)\;dx.
  \end{equation*}
 The natural number $N_{\Gamma,\chi}(\pi)$ denotes the multiplicity of the admissible and irreducible representation $\pi$ in $\widehat{\bigoplus}_{i=0}^\infty V_i/V_{i-1}$. 
\end{satz}

\begin{proof}
  Let $h\in C_c(G)$, such that $\sum_{\gamma\in\Gamma} h(\gamma x)= 1$. 
For $\varphi\in L^2(X,E)$ we compute
\begin{equation*}
  \begin{split}
    R(f)\varphi(x) &= \int_G f(y)\varphi(xy)\;dy\\
                             &= \int_G f(x^{-1}y)\varphi(y)\;dy\\
                             &= \int_G \sum_{\gamma\in\Gamma}h(\gamma^{-1}y)f(x^{-1}y)\varphi(y)\;dy\\
                             &=  \sum_{\gamma\in\Gamma}\int h(y)f(x^{-1}\gamma y)\varphi(\gamma y)\;dy\\
                             &= \int_G h(y)\sum_{\gamma\in\Gamma}f(x^{-1}\gamma y)\chi(\gamma)\varphi(y)\;dy\\
                             &= \int_{\Gamma\backslash G}\sum_{\gamma '\in\Gamma} h(\gamma 'y)\sum_{\gamma\in\Gamma}f(x^{-1}\gamma\gamma ' y)\chi(\gamma)\varphi(\gamma 'y)\;dy\\
                             &= \int_{\Gamma\backslash G}\sum_{\gamma\in\Gamma}f(x^{-1}\gamma y)\chi(\gamma)\varphi(y)\;dy.
                               \end{split}
\end{equation*}

This computation shows that $R(f)$ is an integral operator with integral kernel
\begin{equation*}
k_f(x,y)=\sum_{\gamma\in\Gamma}f(x^{-1}\gamma y)\chi(\gamma).
\end{equation*}
Then we can compute the trace of $R(f)$ by integrating the kernel along the diagonal:

\begin{equation*}
  \tr\; R(f) = \int_X\sum_{\gamma\in\Gamma}f(x^{-1}\gamma x)\tr\; \chi(\gamma)\;dx. 
\end{equation*}

Breaking the integration up into the different conjugacy classes of $G$ we get
\begin{equation*}
  \begin{split}
     \tr R(f) &= \int_{\Gamma\backslash G} \sum_{\gamma '\in\Gamma}h(\gamma ' x)\sum_{\gamma \in\Gamma}f(x^{-1}\gamma x)\tr\;\chi(\gamma)\; dx\\
                    &= h(x)\int_{G} \sum_{\gamma\in\Gamma}f(x^{-1}\gamma x)\tr\;\chi(\gamma)\;dx\\
                    &= \sum_{\gamma\in\Gamma}\int_{G}h(x) f(x^{-1}\gamma x)\tr\;\chi(\gamma)\;dx\\
                    &= \sum_{[\gamma]}\sum_{\sigma\in\Gamma_\gamma\backslash\Gamma}\int_{G}h(\sigma^{-1}x) f( x^{-1}\gamma  x)\tr\;\chi(\gamma)\;dx\\
                    &= \sum_{[\gamma]}\int_{\Gamma_\gamma\backslash G}\sum_{\sigma\in\Gamma_\gamma\backslash\Gamma}\sum_{\eta\in\Gamma_\gamma}h(\sigma^{-1}\eta x) f( x^{-1}\gamma x)\tr\;\chi(\gamma)\;dx\\
                    &= \sum_{[\gamma]}\int_{\Gamma_\gamma\backslash G}f( x^{-1}\gamma x)\tr\;\chi(\gamma)\;dx\\
                     &= \sum_{[\gamma]}\int_{G_\gamma\backslash G}\int_{\Gamma_\gamma\backslash G_\gamma}f( (\sigma x)^{-1}\gamma \sigma x)\tr\;\chi(\gamma)\;d\sigma\;dx\\
                     &= \dim(V_\chi)\text{vol}(\Gamma\backslash G) + \sum_{[\gamma]\neq [1]}\text{vol}(\Gamma_\gamma\backslash G_\gamma)\mathcal{O}_\gamma(f)\tr\;\chi(\gamma).
                        \end{split}
\end{equation*}

On the other hand, according to Proposition \ref{l2filtration} we have a filtration of the representation space $L^2(X,E)$
\begin{equation*}
  0 = V_0 \subset V_1\subset\dots \subset \bigcup_{i=0}^\infty V_i = L^2(X,E).
\end{equation*}
Thus, when $R_i$ is the by $R$ induced representation on $V_i/V_{i-1}$ we get
\begin{equation*}
  \tr\; R(f) = \sum_{i=0}^\infty \tr\; R_i(f),
\end{equation*}
but the right-hand term is obviously equal to 
\begin{equation*}
  \sum_{\pi\in\widehat{G}_{\text{adm}}}N_{\Gamma,\chi}(\pi)\tr\; \pi. 
\end{equation*}

\end{proof}

\bibliographystyle{plain}
\bibliography{nonunitary}

\begin{thebibliography}{1}

\bibitem{prohadeitmar}
Anton Deitmar and Siegfried Echterhoff.
\newblock {\em Principles of harmonic analysis}.
\newblock Universitext. Springer, New York, 2009.

\bibitem{knapp}
Anthony~W. Knapp.
\newblock {\em Representation theory of semisimple groups}.
\newblock Princeton Landmarks in Mathematics. Princeton University Press,
  Princeton, NJ, 2001.
\newblock An overview based on examples, Reprint of the 1986 original.

\bibitem{markus}
A.~S. Markus.
\newblock {\em Introduction to the spectral theory of polynomial operator
  pencils}, volume~71 of {\em Translations of Mathematical Monographs}.
\newblock American Mathematical Society, Providence, RI, 1988.
\newblock Translated from the Russian by H. H. McFaden, Translation edited by
  Ben Silver, With an appendix by M. V. Keldysh.

\bibitem{shubin}
M.~A. Shubin.
\newblock {\em Pseudodifferential operators and spectral theory}.
\newblock Springer-Verlag, Berlin, second edition, 2001.
\newblock Translated from the 1978 Russian original by Stig I. Andersson.

\bibitem{wallach}
Nolan~R. Wallach.
\newblock {\em Real reductive groups. {I}}, volume 132 of {\em Pure and Applied
  Mathematics}.
\newblock Academic Press Inc., Boston, MA, 1988.

\end{thebibliography}
\Addresses
\end{document}